\documentclass[12pt]{amsart}
\usepackage{jrmacros}
\usepackage{thmtools, thm-restate}

\usepackage{mathtools}
\mathtoolsset{showonlyrefs}

\usepackage{amscd}
\usepackage{graphicx}

\theoremstyle{plain}
\newtheorem{theorem}{Theorem}[section]
\newtheorem{proposition}[theorem]{Proposition}

\newtheorem*{theorem*}{Theorem}
\newtheorem*{proposition*}{Proposition}
\newtheorem*{corollary*}{Corollary}
\newtheorem*{lemma*}{Lemma}
\newtheorem*{conjecture*}{Conjecture}

\theoremstyle{definition}
\newtheorem{definition}[theorem]{Definition}

\newtheorem*{definition*}{Definition}
\newtheorem*{example*}{Example}
\newtheorem*{question*}{Question}
\newtheorem*{philosophy*}{Philosophy}

\theoremstyle{remark}
\newtheorem{remark}[theorem]{Remark}

\newtheorem*{remark*}{Remark}

\def\gal{\mathrm{Gal}}

\def\au{\underline{\mathrm{Aut}}^\otimes}
\def\gal{\mathrm{Gal}}
\def\Kb{\overline{K}}

\def\gdr{\gal^{a}}

\def\spa{\mathrm{Sp}}

\def\M{\mathcal{M}}
\def\m{\mathrm{Mod}_R}

\def\ve{\mathrm{Vec}}

\def\sp{\mathrm{Spec}}
\def\M{\cM^0}

\def\Kb{\kb}

\def\Ke{(\sp K)_{\acute{e}t}}

\def\sp{\mathrm{Spec}\,}

\def\PP{p}
\def\QQ{q}
\def\Kb{K^s}

\def\He{H_{\acute{e}t}^0}
\newcommand{\colim}{\operatornamewithlimits{colim}}

\def\m#1{h(#1)}

\begin{document}
\title{A choice-free absolute Galois group and Artin motives}
\author{Julian Rosen}
\date{\today}
\maketitle

\begin{abstract}
Proofs that an arbitrary field has a separable closure are necessarily non-constructive, and separable closures are unique only up to non-canonical isomorphism. This means that the absolute Galois group of a field is defined only up to inner automorphism. Here we construct a profinite algebraic group which is an inner form of the absolute Galois group. Our construction uses no form of the axiom of choice, and the group is defined up to canonical isomorphism. We also show that the Frobenius associated with a prime of a number field unramified in an extension, which is classically defined only up to conjugation, has a uniquely-defined analogue in terms of our group. We give a construction of the category of Artin motives with coefficients in an arbitrary field, and we give an interpretation of our absolute Galois group in terms of this category.
\end{abstract}

\section{Introduction}

Let $L/K$ be a finite Galois field extension. Suppose $G$ is a finite group equipped with an action of $\gal(L/K)$ by group automorphisms, and write $\underline{G}_L$ for $G$ viewed as a constant group scheme over $L$. Then $\gal(L/K)$ acts semilinearly on $\underline{G}_L$, and the quotient $\underline{G}_L/{\gal(L/K)}$ is a finite algebraic group over $K$ whose base change to $L$ is $\underline{G}_L$ (\cite{Ser13}, \S III.1.3). 

\begin{definition}
The \emph{algebraic Galois group of $L/K$} is the finite algebraic group over $K$ given by
\[
\gdr(L/K):=\underline{\gal(L/K)}_L/{\gal(L/K)},
\]
where the action of $\gal(L/K)$ on itself is by conjugation.
\end{definition}

\noindent 
The group of $L$-rational points $\gdr(L/K)(L)$ is identified with $\gal(L/K)$. An explicit description of the coordinate ring of $\gdr(L/K)$ is given by Proposition \ref{propcr} below.

A morphism $\varphi:L_1\to L_2$ of finite Galois extensions of $K$ induces a restriction map $\varphi^*:\gdr(L_2/K)\to\gdr(L_1/K)$. We use these restriction maps to define the algebraic absolute Galois group as a categorical limit.

\begin{definition}
\label{defagg}
The \emph{algebraic absolute Galois group} of $K$ is the profinite algebraic group over $K$ defined by
\begin{equation}
\label{eqagg}
\gdr(K):=\lim_{\substack{L/K\\\text{fin.\ Gal.}}}\gdr(L/K),
\end{equation}
where the limit is taken over the essentially small category of finite Galois extensions of $K$.
\end{definition}

Definition \ref{defagg} does not invoke any form of the Axiom of Choice. By contrast, the usal constructions of separable closure \emph{do} use Choice. While Banaschewski \cite{Ban92} has constructed a separable closure using only the Boolean Ultrafilter Theorem (which is strictly weaker than Choice), Pincus \cite{Pin72} has shown that it is consistent with ZF that there is a field with no separable closure.

\begin{remark}
If $\gdr(L/K)$ is replaced by $\gal(L/K)$ on the right hand side of \eqref{eqagg}, we do not recover the usual absolute Galois group but rather its center, which for many fields of interest is trivial. 
\end{remark}

\subsection{Results}
The following result shows that $\gdr(K)$ is an algebraic version of the ordinary absolute Galois group.
\begin{restatable}{Theorem}{thagg}
\label{thagg}
For each separable closure $\Kb$ of $K$, there is a distinguished isomorphism 
\begin{equation}
\label{eqiso}
\gal(\Kb/K)\xrightarrow{\sim}\gdr(K)(\Kb).
\end{equation}
The isomorphism \eqref{eqiso} is functorial in $\Kb$, in the sense that if $\hat{\Kb}$ is another separable closure and $\varphi:\Kb\to\hat{\Kb}$ is a morphism, then the following diagram commutes:
\begin{equation}
\label{eqCD}
\begin{CD}
\gal(\Kb/K)@>\sim>>\gdr(K)(\Kb) \\
@V{\sigma\mapsto \varphi\sigma\varphi^{-1}}VV @VV\varphi V\\
\gal(\hat{\Kb}/K) @>\sim>>\gdr(K)(\hat{\Kb}) .
\end{CD}
\end{equation}
\end{restatable}

The main tool in the proof of Theorem \ref{thagg} is the following result, which is interesting in its own right.
\begin{restatable}{Theorem}{thuniq}
\label{thuniq}
If $\varphi$, $\psi: L_1\to L_2$ are homomorphisms of finite Galois extensions of $K$, then $\varphi^*=\psi^*:\gdr(L_2/K)\to\gdr(L_1/K)$. In particular, if $L_1$ and $L_2$ are non-canonically isomorphic, then $\gdr(L_1/K)$ and $\gdr(L_2/K)$ are canonically isomorphic.
\end{restatable}

\noindent For example, while the Galois group of a separable polynomial is defined only up inner automorphism, the algebraic Galois group of a polynomial is defined up to canonical isomorphism.

Now let $L/K$ be a Galois extension of number fields. Let $\PP$ be a prime of $K$ unramified in $L$, and $\QQ$ a prime of $L$ over $K$. The Frobenius automorphism is an element $\phi_\QQ\in\gal(L/K)$. If $\QQ'$ is another prime of $L$ over $\PP$, then $\phi_{\QQ'}$ is conjugate to $\phi_{\QQ}$. So if we wish to talk about a Frobenius associated with $\PP$, we must consider a conjugacy class in $\gal(L/K)$. The algebraic analogue of the Frobenius is better behaved.
\begin{restatable}{Theorem}{thnf}
\label{thnf}
Let $L/K$ be a Galois extension of number fields, $\PP$ a finite or infinite place of $K$ unramified in $L$, and write $K_\PP$ for the completion of $K$ at $\PP$. Then there is a unique element
\[
\phi_p^a\in\gdr(L/K)(K_\PP)
\]
with the property that for every place $\QQ$ of $L$ over $\PP$, the image of $\phi_\PP^a$ in $\gdr(L/K)(K_\PP)\hookrightarrow\gdr(L/K)(L_\QQ)$ is equal to the image of $\phi_\QQ$ in $\gal(L/K)=\gdr(L/K)(L)\hookrightarrow\gdr(L/K)(L_\QQ)$.
\end{restatable}

\subsection{Artin motives} The algebraic Galois group has an interpretation in terms of Artin motives. For each field $K$, there is an abelian tensor category $\M_{K,K}$, the category of Artin motives over $K$ with coefficients in $K$. We give a construction of $\M_{K,K}$ in Sec.\ \ref{secam}. Each scheme $X$ finite \'etale over $\sp K$ determines an object $\m{X}\in\M_K$, the motive of $X$. The category $\M_{K,K}$ is neutral Tannakian, which means that if $\omega$ is a fibre functor on $\M_{K,K}$ (that is, $\omega$ is an exact $K$-linear tensor functor from $\M_{K,K}$ to finite dimensional $K$-vector spaces), then $\omega$ induces a tensor equivalence of $\M_{K,K}$ with the category of finite dimensional representations of an affine group scheme $\au(\omega)$. If $\Kb$ is a separable closure of $K$, there is a fibre functor $\omega_{\Kb}$ taking $\m{X}$ to the \'etale cohomology group
\[
\He(X_{\Kb},K)=K^{X(\Kb)}.
\] The automorphism group of $\omega_{\Kb}$ is $\gal(\Kb/K)$ (viewed as a constant profinite group scheme over $K$), so $\M_{K,K}$ is equivalent to the category of finite-dimensional continuous $\gal(\Kb/K)$-representations on discrete $K$-vector spaces.

To construct the functor $\omega_{\Kb}$ one must choose a separable closure of $K$. There is another fibre functor  on $\M_{K,K}$ whose construction does not require choice, the de Rham realization $\omega_{dR}$, which satisfies
\[
\omega_{dR}(\m{X})=H^0_{dR}(X):=\Gamma(X,\cO_X).
\]
We prove the following result.
\begin{restatable}{Theorem}{tham}
\label{tham}
The automorphism group $\au(\omega_{dR})$ is isomorphic to $\gdr(K)$, so $\omega_{dR}$ induces an equivalence of $\M_{K,K}$ with the category of finite-dimensional algebraic representations of $\gdr(K)$.
\end{restatable}

\section{Coordinate ring}
We begin with a description of the coordinate ring of $\gdr(L/K)$.

\begin{proposition}
\label{propcr}
Let $L/K$ be a finite Galois extension. Then the coordinate ring of $\gdr(L/K)$ is isomorphic to the ring
\[
A(L/K):=\left\{\begin{array}{rrr}f:\gal(L/K)\to L&\big|& \sigma f(\sigma^{-1}\tau \sigma)=f(\tau)\\&\big|&\forall\sigma,\tau\in\gal(L/K)\end{array}\right\}.
\]
\end{proposition}
\begin{proof}
More generally, if $X$ is an affine variety acted on by a finite group $G$, the coordinate ring of $X/G$ is the set of $G$-invariant elements in the coordinate ring of $X$. In our case, the coordinate ring of $\underline{\gal(L/K)}_L$ is the ring of functions $\gal(L/K)\to L$.
\end{proof}
Now we prove Theorem \ref{thuniq}. We recall the statement here.

\thuniq*

\begin{proof}
Let $L_1/K$ and $L_2/K$ be Galois extensions, $\varphi:L_1\to L_2$ a map of extensions. First we construct $\varphi^*:\gdr(L_2/K)\to\gdr(L_1/K)$, or equivalently a morphism of coordinate rings $\varphi_*:A(L_1/K)\to A(L_2/K)$. Given $f\in A(L_1/K)$, we define
\[
\varphi_*f(\tau) =\varphi\circ f\lp \tau\big|_{\varphi,L_1}\rp.
\]
Here $\tau\big|_{\varphi,L_1}$ denotes the restriction of $\tau$ to $L_1$ via $\varphi$. One immediately checks that $\varphi_*f\in A(L_2/K)$, and that $\varphi\mapsto\varphi_*$ is functorial.

Now let $\psi:L_1\to L_2$ be another map of extensions. There is an element $\sigma\in\gal(L_1/K)$ such that $\varphi\circ\sigma=\psi$. Then for arbitrary $f\in\gal(L_1/K)$ and $\tau\in\gal(L_2/K)$:
\begin{align*}
\psi_*f(\tau)&=\psi\circ f\lp \tau\big|_{\psi,L_1}\rp\\
&=\varphi\circ\sigma\circ f\lp \tau\big|_{\varphi\sigma,L_1}\rp\\
&=\varphi\circ\sigma\circ f\lp \sigma^{-1}\tau\big|_{\varphi,L_1}\sigma\rp\\
&=\varphi\circ f\lp \tau\big|_{\varphi,L_1}\rp\\
&=\varphi_*f(\tau).
\end{align*}
Thus we have $\varphi_*=\psi_*$.
\end{proof}

\section{Algebraic absolute Galois group}
The (ordinary) absolute Galois group is a projective limit of finite Galois groups. However, the projective limit must be taken over finite subextensions of some fixed separable closure:
\[
\gal(\Kb/K)=\varprojlim_{\substack{L/K\text{ fin.\ Gal.}\\L\subset \Kb}}\gal(L/K).
\]
If the condition $L\subset\Kb$ is dropped and we instead take a categorical limit over all finite Galois extensions $L/K$, we obtain instead the center of $\gal(\Kb/K)$, which in many cases is trivial e.g.\ if $K$ is a number field (\cite{Win08}, Corollary 12.1.6).

Now we prove Theorem \ref{thagg}. We recall the statement here.
\thagg*
\begin{proof}
Let $\Kb$ be a separable closure of $K$. For each finite Galois extension $L/K$, let $\tilde{L}$ be the image of $L$ under one (equivalently, every) embedding $L\hookrightarrow\Kb$. Then Theorem \ref{thuniq} implies $\gdr(L/K)$ and $\gdr(\tilde{L}/K)$ are isomorphic, and this isomorphism is functorial in $L$. Now
\begin{align*}
\gdr(K)(\Kb)&=\lim_L\gdr(L/K)(\Kb)\\
&=\varprojlim_{\tilde{L}\subset\Kb}\gdr(\tilde{L}/K)(\Kb)\\
&=\varprojlim_{\tilde{L}\subset\Kb}\gal(\tilde{L}/K)\\
&=\gal(\Kb/K).
\end{align*}
We can describe this identification on the level of coordinate rings. Given $\sigma\in\gal(\Kb/K)$, we get an algebra homomorphism
\begin{align*}
g_{\sigma}:\Gamma(\cO_{\gdr(K)})=\colim_L A(L/K)\to \Kb,
\end{align*}
which for fixed $L$, takes $f\in A(L/K)$ to $\psi\circ f(\sigma\big|_{\psi,L})$, where $\psi:L\to \Kb$ is any embedding (one checks that the result is independent of $\psi$).

Now, let $\hat{\Kb}$ be another separable closure of $K$ and $\varphi:\Kb\to\hat{\Kb}$. Fix $\sigma\in\gal(\Kb/K)$, a finite Galois extension $L/K$, and an embedding $\psi:L\to\Kb$. Then $\varphi\circ\psi$ is an embedding $L\to\hat{\Kb}$. We get two elements of $\gdr(K)(\hat{\Kb})$, and two show that \eqref{eqCD} commutes we need to show these two elements are equal. In \eqref{eqCD}, mapping down first then to the right gives
\begin{align*}
g_{\varphi\sigma\varphi^{-1}}:\Gamma(\cO_{\gdr(K)})&\to\hat{\Kb},\\
f\in A(L/K)&\mapsto \varphi\circ\psi\circ f(\varphi\sigma\varphi^{-1}\big|_{\varphi\psi,L}).
\end{align*}
Mapping to the right first then down gives
\begin{align*}
\varphi\circ g_{\sigma}:\Gamma(\cO_{\gdr(K)})&\to\hat{\Kb},\\
f\in A(L/K)&\mapsto \varphi\circ\psi\circ f(\sigma\big|_{\psi,L}).
\end{align*}
This completes the proof, as
\[
\varphi\sigma\varphi^{-1}\big|_{\varphi\psi,L}=\sigma\big|_{\psi,L}.
\]
\end{proof}

\section{Number fields}
In this section we prove Theorem \ref{thnf}. We recall the statement here.

\thnf*

\begin{proof}
Choose a place $\QQ$ of $L$ over $\PP$. We define a $K$-algebra homomorphism
\begin{align*}
g_\QQ:A(L/K)&\to L_\QQ,\\
f&\mapsto f(\phi_{\QQ}).
\end{align*}
The Frobenius $\phi_{\QQ}$ extends uniquely to a generator the cyclic group $\gal(L_\QQ,K_\PP)$, and we have
\begin{align*}
\phi_{\QQ}\circ f(\phi_{\QQ})&=f\big(\phi_{\QQ}\phi_{\QQ}\phi_{\QQ}^{-1}\big)\\
&=f(\phi_{\QQ}).
\end{align*}
It follows that $f(\phi_{\QQ})\in K_p$, so the the image of $g_{\QQ}$ is contained in $K_\PP$. Write $i_\PP$ for the inclusion $K_\PP\hookrightarrow L_\QQ$. Then $\phi^a_\QQ:=i_\QQ^{-1}\circ f_{\QQ}$ is the unique element of $\gdr(L/K)(K_\PP)$ satisfying the conclusion of the theorem for $\QQ$.

To complete the proof of the theorem, we will show that if $\QQ'$ is another place of $L$ over $\PP$, then $\phi^a_{\QQ'}=\phi^a_{\QQ}$. We can then define $\phi^a_{\PP}:=\phi^a_{\QQ}$, which depends only on $\PP$. There is some $\sigma\in\gal(L/K)$ such that $\QQ'=\QQ^\sigma$. Then $\phi_{\QQ'}=\sigma\phi_\QQ\sigma^{-1}$. Additionally, $\sigma$ induces an isomorphism $\bar{\sigma}:L_\QQ\to L_{\QQ'}$, and $\bar{\sigma}\circ i_\QQ=i_{\QQ'}$. We have
\begin{align*}
\phi^a_{\QQ'}(f)&=i_{\QQ'}^{-1}\circ f(\phi_{\QQ'})\\
&=i_\QQ^{-1}\circ\bar{\sigma}^{-1}\circ f\big(\sigma\phi_\QQ\sigma^{-1}\big)\\
&=i_\QQ^{-1}\circ f(\phi_\QQ)=\phi^a_\QQ(f).
\end{align*}
This competes the proof.
\end{proof}

\section{Artin motives}
\label{secam}

\noindent Fix a base field $K$ and a coefficient field $F$.

\subsection{Construction of the category}
We begin with choice-free construction of the category $\M_{K,F}$ of Artin motives over $K$ with coefficients in $F$. In Sec.\ \ref{ssdr} we use this category to give a another construction of the algebraic absolute Galois group.

Recall that the small \'etale site over $\sp K$, denoted $\Ke$, is the category of finite \'etale covers of $\sp K$ (that is, finite unions of spectra of finite separable extensions of $K$), equipped with a Grothendieck topology such that a family is covering if it is jointly surjective.  If $\cF$ is a sheaf on $\Ke$, we write $\cF(L)$ for $\cF(\sp L)$.

We define $\M_{K,F}$ to be the category of sheaves of $F$-vector spaces on $\Ke$ satisfying a finiteness condition which we now give.

\begin{definition}
\label{deft}
\begin{enumerate}
\item A sheaf $\cF$ of $F$-vector spaces on $\Ke$ is \emph{finite type} if there exists a finite separable extension $L/K$ such that $\cF(L)$ has finite dimension and $\cF(L)\to\cF(L')$ is an isomorphism for every finite separable $L'/L$.
\item The \emph{category of Artin motives over $K$ with coefficients in $F$}, denoted $\M_{K,F}$, is the category of finite type sheaves of $F$-vector spaces on $\Ke$. 
\item For each scheme $X$ finite \'{e}tale over $\sp K$, the \emph{motive of $X$} is the sheaf $\m{X}\in M_{K,F}$ given by
\[
\m{X}(Y):=F^{\spa(X\times Y)},
\]
where $\spa(X\times Y)$ is the underlying set of the scheme $X\times_{\sp K} Y$ (which is a finite discrete set).
\end{enumerate}
\end{definition}

\begin{proposition}
Let $X$ be a finite \'{e}tale scheme over $\sp K$. Then the $h(X)$ defined by Definition \ref{deft}(3) is a finite type sheaf.
\end{proposition}
\begin{proof}
The condition that $\m{X}$ is a sheaf amounts to the identification
\[
\spa(X_{L'})/\gal(L'/L)\cong\spa(X_L)
\]
for $L'/L$ finite Galois, and $\m{X}$ is finite type because in Definition \ref{deft}(1) we may take $L$ to be any finite Galois extension of $K$ into which the residue fields of all points of $X$ embed.
\end{proof}

We consider fibre functors on $\M_{K,F}$, which are exact, $F$-linear tensor functors from $\M_{K,F}$ to $\ve_F$, the category of finite-dimensional $F$-vector spaces. Each fibre functor $\omega$ determines an affine group scheme $\au(\omega)$ over $F$, whose $R$-points (for an $F$-algebra $R$) are the $R$-linear tensor functorial automorphism of the functor
\[
A\mapsto\omega(A)\otimes_F R.
\]
The category $M_{K,F}$ is neutral Tannakian (this is a consequence of Proposition \ref{propequiv} below), which means that each fibre functor $\omega$ induces an equivalence of $\M_{K,F}$ with the category of finite dimensional algebraic representations of $\au(\omega)$. Our fibre functors come from cohomological realizations: if $H^0_\bullet$ is a cohomology theory for $0$-dimensional varieties with coefficients in $F$, the corresponding fibre functor takes $h(X)$ to $H^0_\bullet(X)$.

\subsection{\'{E}tale realization}

\begin{definition}
For $\Kb$ a separable closure of $K$, we define a fibre functor
\begin{align*}
\omega_{\Kb}:\M_{K,F}&\to\ve_F,\\
\cF&\mapsto\varinjlim_{L\subset \Kb}\cF(L),
\end{align*}
where the limit is over subfields of $\Kb$ finite over $K$.
\end{definition}
Because $\cF$ is finite type, we have $\omega_{\Kb}(\cF)=\cF(L)$ for any $L\subset\Kb$ satisfying Definition \ref{deft}(1), so in particular $\omega_{\Kb}(\cF)$ is finite dimensional. Additionally, there is an action of $\gal(\Kb/K)$ on $\omega_{\Kb}(\cF)$ which factors through $\gal(L/K)$ for every $L\subset \Kb$ finite Galois over $K$. Note that while the construction of the category $\M_{K,F}$ did not use the axiom of choice, some form of choice is needed to construct a separable closure $\Kb$, hence also to construct the fibre functor $\omega_{\Kb}$.

The functor $\omega_{\Kb}$ is an \'{e}tale realization, as the following proposition shows.
\begin{proposition}
There is a $\gal(\Kb/K)$-equivariant isomorphism
\[
\omega_{\Kb}(\m{X})\cong \He(X_{\Kb},F),
\]
functorial in $X\in\M_{K,F}$.
\end{proposition}
\begin{proof}
We compute
\begin{align*}
\omega_{\Kb}(\m{X})&=\varinjlim_{L\subset\Kb}F^{\spa(X_L)}\\
&=F^{\spa(X_{\Kb})}\\
&=F^{X(\Kb)}\\
&=\He(X_{\Kb},F),
\end{align*}
where all equalities are $\gal(\Kb/K)$-equivariant and functorial in $X$. Here $X_{\Kb}$ is the base change of $X$ to $\Kb$.
\end{proof}

\begin{proposition}
\label{propequiv}

The functor $\omega_{\Kb}$ induces an equivalence of tensor categories between $\M_{K,F}$ and the category of finite dimensional discrete $F$-vector spaces equipped with a continuous action of $\gal(\Kb/K)$.
\end{proposition}
\begin{proof}
It is known (\cite{Dej13}, Theorem 55.3) that the functor
\[
\cF\mapsto\varinjlim_{L\subset\Kb}\cF(L)
\]
induces an equivalence betwee the category of sheaves of sets on $\Ke$ and the category of sets equipped with a continuous action of $G=\gal(\Kb/K)$. The inverse is given by
\begin{equation}
\label{eqinv}
S\mapsto\bigg(\cF_S:U\mapsto\hom_{G}(U(\Kb),S)\bigg).
\end{equation}
It follows that we also get an equivalence between sheaves of $F$-vector spaces on $\Ke$ and continuous $\gal(\Kb/K)$-representations on discrete $F$-vector spaces. 

A discrete $\gal(\Kb/K)$-representation $V$ is finite dimensional if and only there is some finite Galois $L/K$ such that $V$ factors through a finite dimensional representation of $\gal(L/K)$. From \eqref{eqinv}, we see that this happens if and only if $\cF_V(L)$ is finite dimensional and the map $\cF_V(L)\to\cF_V(L')$ is an isomorphism for every $L'/L$.
\end{proof}

Proposition \ref{propequiv} implies that $\M_{K,F}$ is a neutral Tannakian category, as it is equivalent to the finite dimensional representations of a group. The proposition implies $\au(\omega_{\Kb})$ is isomorphic to $\gal(\Kb/K)$, viewed as a constant profinite group scheme over $F$. 

Additionally, Proposition \ref{propequiv} implies that the Artin motives we define are equivalent to other notions of Artin motives in the literature. If $K$ has characteristic $0$, then $\M_{K,\Q}$ is equivalent to the category of absolute Hodge $0$-motives (\cite{Del82}, Proposition 6.17). For $K\subset\C$, we get that $\M_{K,\Q}$ is equivalent to the category of Nori $0$-motives (\cite{Hub17}, \S9.4).

\subsection{De Rham cohomology}
\label{ssdr}
Here we specialize to the case $F=K$. In this case there is another cohomology theory, algebraic de Rham cohomology, defined by
\[
H^0_{dR}(X):=\Gamma(X,\cO_X).
\]
Algebraic de Rham cohomology induces a fibre functor $\omega_{dR}$ on $\M_{K,K}$, which does not require constructing a separable closure of $K$.

\begin{definition}
We define the \emph{de Rham realization functor}
\begin{align*}
\omega_{dR}:\M_{K,K}&\to\ve_K,\\
\cF&\mapsto\colim_{L\in\Ke}(\cF(L)\otimes L)^{\gal(L/K)}.
\end{align*}
\end{definition}
\begin{proposition}
For $X\in\Ke$, there is an isomorphism
\[
\omega_{dR}(\m{X})\cong H^0_{dR}(X),
\]
which is functorial in $X$.
\end{proposition}
\begin{proof}
We say a $0$-dimensional variety $X$ splits over an extension $L/K$ is $X_L$ is a discrete. The class of fields over which $X$ splits is cofinal among the finite separable extensions of $K$. We have
\begin{align*}
\colim_L\big(\m{X}(L)\otimes L\big)^{\gal(L/K)}&=\colim_L\big(L^{\spa(X_L)}\big)^{\gal(L/K)}\\
&=\colim_{\substack{L\\X\text{ splits over }L}}\big(L^{\spa(X_L)}\big)^{\gal(L/K)}\\
&=\colim_{\substack{L\\X\text{ splits over }L}}\big(L^{X(L)}\big)^{\gal(L/K)}\\
&=\colim_{\substack{L\\X\text{ splits over }L}}\big\{\gal(L/K)\text{-equivariant functions }X(L)\to L\big\}\\
&=\colim_{\substack{L\\X\text{ splits over }L}}\Gamma(X,\cO_X)\\
&=\Gamma(X,\cO_X).
\end{align*}
One checks that each equality above is functorial in $X$.
\end{proof}
 
 Finally, we prove Theorem \ref{tham}. We recall the statement here.
 \tham*
 
\begin{proof}
For each finite Galois extension $L/K$, define
\[
q_L(\cF):=(\cF(L)\otimes L)^{\gal(L/K)}.
\]
One checks that a morphism $\varphi:L_1\to L_2$ of $K$-extensions induces a $K$-linear map $\varphi_*:q_{L_1}(\cF)\to q_{L_2}(\cF)$, and that is $\psi:L_1\to L_2$ is another map of extensions, then $\varphi_*=\psi_*$ (the verification is similar to the proof of Theorem \ref{thuniq}). Thus for any separable closure $\Kb$ of $K$, we have
\begin{align*}
\omega_{dR}(\cF)&=\varinjlim_{L\subset \Kb}(\cF(L)\otimes L)^{\gal(L/K)}\\
&=(\omega_{\Kb}(\cF)\otimes \Kb)^{\gal(\Kb/K)}.
\end{align*}
By \cite{Del82}, Theorem 3.2, fibre functors on $\M_{K,K}$ correspond to $\gal(\Kb/K)$-torsors over $K$, and $\omega_{dR}$ corresponds to the torsor $\Kb$. Now, as a $\gal(\Kb/K)$-torsor, $\Kb$ corresponds to the Galois cohomology class in $H^1(K,\gal(\Kb/K))$ coming from the identity cocycle. It follows that $\au(\omega_{dR})$ is the inner form of $\gal(\Kb/K)$ corresponding to the class in $H^1(K,\mathrm{Inn}(\gal(\Kb/K)))$ whose cocycle is the natural map $\gal(\Kb/K)\to\mathrm{Inn}(\gal(\Kb/K))$. In other words, $\au(\omega_{dR})$ is isomorphic to the quotient
\[
\underline{\gal(\Kb/K)}_{\Kb}/\gal(\Kb/K),
\]
where $\gal(\Kb/K)$ acts semi-linearly on $\underline{\gal(\Kb/K)}_{\Kb}$ by conjugation. This shows that $\au(\omega_{dR})$ is isomorphic to $\gdr(K)$.
\end{proof}

\ack We thank Danny Krashen for several intersting discussions.

\bibliographystyle{hplain}
\bibliography{jrbiblio}

\newcommand{\noop}[1]{} \def\cprime{$'$}
\begin{thebibliography}{1}

\bibitem{Ban92}
Bernhard Banaschewski.
\newblock Algebraic closure without choice.
\newblock {\em Mathematical Logic Quarterly}, 38(1):383--385, 1992.

\bibitem{Dej13}
Aise~Johan De~Jong et~al.
\newblock The stacks project.
\newblock {\em \'{E}tale cohomology}, 2013.

\bibitem{Del82}
Pierre Deligne and James Milne.
\newblock Tannakian categories.
\newblock In {\em Hodge cycles, motives, and {S}himura varieties}, volume 900
  of {\em Lecture Notes in Mathematics}, pages 101--228. Springer-Verlag,
  Berlin-New York, 1982.

\bibitem{Hub17}
Annette Huber and Stefan M{\"u}ller-Stach.
\newblock {\em Periods and Nori Motives}, volume~65 of {\em Ergebnisse der
  Mathematik und ihrer Grenzgebiete}.
\newblock Springer Berlin Heidelberg, 2017.

\bibitem{Pin72}
David Pincus.
\newblock Zermelo-fraenkel consistency results by fraenkel-mostowski methods.
\newblock {\em The Journal of Symbolic Logic}, 37(04):721--743, 1972.

\bibitem{Ser13}
Jean-Pierre Serre.
\newblock {\em Galois cohomology}.
\newblock Springer Science \& Business Media, 2013.

\bibitem{Win08}
Kay Wingberg, Alexander Schmidt, and J{\"u}rgen Neukirch.
\newblock {\em Cohomology of number fields}.
\newblock Springer Berlin Heidelberg, 2008.

\end{thebibliography}
\end{document}